\documentclass[11pt]{amsart}
\usepackage[latin1]{inputenc}
\usepackage[T1]{fontenc}
\usepackage{hyperref}
\usepackage{amsfonts,dsfont}
\usepackage{amsmath}
\usepackage{epsf, subfigure, verbatim}
\usepackage{amssymb}
\usepackage{amsthm}
\usepackage{latexsym}
\usepackage{layout}
\usepackage{amsrefs}
\usepackage{color}

\usepackage[normalem]{ulem}

\usepackage[titletoc]{appendix}

\newtheorem{theorem}{Theorem}

\newtheorem{prop}{Proposition}[section]

\newtheorem{remark}{Remark}[section]

\newcommand{\ee}{\mathrm{e}} 
\newcommand{\dd}{\mathrm{d}} 
\newcommand{\ind}{\mathds{1}}
\newcommand{\e}{\mathbb{E}}
\newcommand{\p}{\mathbb{P}}
\newcommand{\drift}{c}
\newcommand{\parisiantime}{\kappa_q}
\newcommand{\cumparisiantime}{\sigma_r}

\newcommand{\pscale}{W^{(p)}}

\newcommand{\wapq}{\mathcal W_a^{(p,q)}}

\newcommand{\cH}{\mathcal{H}}

\newcommand{\dE}{\mathbb{E}}
\newcommand{\dP}{\mathbb{P}}

\newcommand{\dy}{\mathrm{d}y}

\newcommand{\ABS}[1]{{\left| #1 \right|}} 
\newcommand{\PAR}[1]{{\left(#1\right)}} 
\newcommand{\SBRA}[1]{{\left[#1\right]}} 
\newcommand{\BRA}[1]{{\left\{#1\right\}}} 

\begin{document}

\title[]{On the distribution of cumulative Parisian ruin}

\author[]{H\'el\`ene Gu\'erin}
\address{IRMAR, Universit\'e de Rennes 1, Campus de Beaulieu 35042 Rennes Cedex France}
\email{helene.guerin@univ-rennes1.fr}

\author[]{Jean-Fran\c{c}ois Renaud}
\address{D\'epartement de math\'ematiques, Universit\'e du Qu\'ebec \`a Montr\'eal (UQAM), 201 av.\ Pr\'esident-Kennedy, Montr\'eal (Qu\'ebec) H2X 3Y7, Canada}
\email{renaud.jf@uqam.ca}

\date{\today}

\keywords{Cumulative Parisian ruin, occupation time, Cram\'er-Lundberg model, finite-time ruin.}
\subjclass[2000]{}

\begin{abstract}
We introduce the concept of cumulative Parisian ruin, which is based on the time spent in the red by the underlying surplus process. Our main result is an explicit representation for the distribution of the occupation time, over a finite-time horizon, for a compound Poisson process with drift and exponential claims. The Brownian ruin model is also studied in details. Finally, we analyze for a general framework the relationships between cumulative Parisian ruin and classical ruin, as well as with Parisian ruin based on exponential implementation delays. 
\end{abstract}

\maketitle

\section{Introduction}

Historically, in classical actuarial ruin theory, the central topic has been the analysis of the probability of ruin, which is a measure to assess the overall risk of an insurance company. Ruin occurs when the surplus process falls below a certain threshold level for the first time, e.g.\ the insurer's minimum capital requirement set by the regulatory body. Then, more sophisticated ruin-based risk measures using the timing and the severity of such a capital shortfall have been proposed, culminating with so-called Gerber-Shiu distributions; see \cites{gerbershiu1998,kyprianou2013}.

However, as discussed in \cite{dosreis1993}: "\textit{[...] sometimes the event ruin has a very small probability and the portfolio is just one out of many existing in the company. The company can have enough funds available to support (or ask for external support) some negative surplus for some time, in the hope that the portfolio will recover in the future, allowing the company to keep this business alive. This can be regarded as an investment, since the process will recover in the future.}" This is why, in recent years, new ruin models, incorporating this realistic feature that a company is not immediately liquidated when it defaults, have been proposed. Therefore, new risk measures and models have been introduced: Parisian ruin (see, e.g., \cites{czarnapalmowski2010,loeffenetal2013,landriaultetal2014,baurdoux_et_al_2015}), random observations (see, e.g., \cites{albrecheretal2011a,albrecheretal2013}) and Omega models (see, e.g., \cite{albrecheretal2011b}, \cite{gerber_et_al_2012} and \cite{albrecher_lautscham_2013}). In those papers, new definitions of \textit{ruin}, \textit{bankruptcy} and/or \textit{liquidation} are proposed and studied.

In this direction, we will introduce the concept of cumulative Parisian ruin, which is based on the time spent in the red by the underlying surplus process. While Parisian ruin has been inspired by Parisian options, the idea of cumulative Parisian ruin comes from cumulative Parisian options, another type of occupation-time options. In the definition of Parisian ruin (the duration of) periods of financial distress are considered separately, while cumulative Parisian ruin considers the aggregation of those durations. As expected, this notion of ruin is closely related to the occupation-time process of the underlying surplus process in the model. We will also analyze the relationships between cumulative Parisian ruin and classical ruin, as well as with Parisian ruin based on exponential implementation delays. 

\subsection{Insurance risk models}

In this paper, we will mainly consider the following two classical actuarial ruin models: the Cram\'er-Lundberg model with exponential claims and the Brownian ruin model. First, in the Cram\'er-Lundberg model, the surplus of an insurance company/portfolio is modelled by a compound Poisson process with drift $X^\text{CL} = \{X^\text{CL}_t, t \geq 0\}$, that is
\begin{equation}\label{E:cramerlundberg}
X^\text{CL}_t = X^\text{CL}_0 + \drift t - \sum_{i=1}^{N_t} C_i ,
\end{equation}
where $\drift>0$ denotes the constant premium rate, $N = \{N_t, t \geq 0\}$ is a Poisson process with rate $\lambda>0$ and the $C_i$'s form a sequence of independent and identically distributed random variables (also independent of $N$) representing the claim amounts; each $C_i$ is assumed to follow an exponential distribution with rate $\alpha>0$. In this model, the random variable $\sum_{i=1}^{N_t} C_i$ represents the aggregate claim payments made up to time $t$. 

On the other hand, in the Brownian ruin model, the surplus of the insurance company is modelled by a Brownian motion with drift $X^\text{BM} = \{X^\text{BM}_t, t \geq 0\}$, that is
\begin{equation}\label{E:Brownian}
X^\text{BM}_t = X^\text{BM}_0 + \drift t + \sigma B_t ,
\end{equation}
where $\drift>0$ denotes the constant premium rate, $\sigma>0$ and $B = \{B_t, t \geq 0\}$ is a standard Brownian motion.

Note that both surplus processes $X^\text{CL}$ and $X^\text{BM}$ are special cases of a spectrally negative L\'{e}vy process (SNLP), also known as a general L\'evy insurance risk process, that is a càdlàg process $X = \{X_t, t \geq 0\}$ with stationary and independent increments and no positive jumps. There is a rich literature on this family of stochastic processes for tackling \textit{classical} ruin problems; see e.g.\ \cites{kyprianou2013,kyprianou2014}.

\subsection{Definitions of ruin and bankruptcy}

Let $X=\{X_t, t\geq 0\}$ be a surplus process. In what follows, the law of $X$ with initial capital $X_0 = x$ is denoted by $\p_x$ and the corresponding expectation by $\e_x$. We write $\p$ and $\e$ when $x=0$.

The time of \textit{classical} ruin is defined as
$$
\tau_0^- = \inf \left\lbrace t > 0 \colon X_t < 0 \right\rbrace ,
$$
yielding the classical ruin-based risk measure $\p_x \left( \tau_0^- < \infty \right)$, i.e.\ the probability of ruin. From a practical point of view, it is more interesting to consider the finite-time probability of ruin $\p_x \left( \tau_0^- \leq t \right)$, where $t>0$ is a fixed deterministic time corresponding to the time horizon of interest for the risk managers.

In the last few years, in order to distinguish between default and bankruptcy, several concepts of \textit{ruin} have been proposed. In particular, the idea of Parisian ruin has attracted a lot of attention; it considers the application of an implementation delay in the recognition of an insurer's capital insufficiency. More precisely, it assumes that Parisian ruin occurs if the excursion below the critical threshold level $0$ is longer than a pre-determined time called the implementation delay, or the \textit{clock}. Two types of Parisian ruin have been considered: with deterministic delays or with stochastic delays.

First, in \cites{czarnapalmowski2010, loeffenetal2013}, a Parisian ruin time (with a deterministic delay $r>0$) is defined as
$$
\tau_r = \inf \left\lbrace t > 0 \colon t - g_t > r \right\rbrace ,
$$
where $g_t = \sup \left\lbrace 0 \leq s \leq t \colon X_s \geq 0 \right\rbrace$, and studied for L\'evy insurance risk models. In other words, the company is said to be \textit{Parisian ruined} the first time an excursion below zero lasts longer than the fixed implementation delay $r$. We will call this \textit{deterministic Parisian ruin}, where \textit{deterministic} refers to the implementation delay.

On the other hand, in \cites{landriaultetal2011,landriaultetal2014,baurdoux_et_al_2015}, exponentially distributed delays have been considered. A descriptive definition of Parisian ruin with exponential delay is given in \cite{baurdoux_et_al_2015}, for spectrally negative L\'evy processes. The idea is to mark the Poisson point process of excursions away from $0$ by independent copies of a generic exponential random variable $\mathbf{e}_q$ with rate $q>0$. Then, the Parisian ruin time is defined as
$$
\parisiantime = \inf \left\lbrace t > 0 \colon t - g_t > \mathbf{e}_q^{g_t} \right\rbrace ,
$$
where $g_t = \sup \left\lbrace 0 \leq s \leq t \colon X_s \geq 0 \right\rbrace$ and where each random variable $\mathbf{e}_q^{g_t}$ is exponentially distributed with rate $q>0$. We will call this \textit{exponential Parisian ruin}.

Note that for the two definitions of Parisian ruin mentioned above, there is a different clock for each excursion, either deterministic or stochastic.

\begin{remark}
Note that the definition of exponential Parisian ruin corresponds to a specific choice of function $\omega(\cdot)$ in an Omega model; see \cites{gerber_et_al_2012,albrecher_lautscham_2013,loeffenetal2014}.
\end{remark}

The rest of the paper is organized as follows. In Section 2, after introducing the concept of cumulative Parisian ruin, we derive our main results. More precisely, in Section 2.1 and in Section 2.2, we obtain the distribution of cumulative Parisian ruin when the underlying surplus process is a Brownian motion risk process and a Cramér-Lundberg process with exponential claims, respectively. Section 2.3 is devoted to the study of the relationships between cumulative Parisian and classical ruin, as well as with exponential Parisian ruin.

\section{Cumulative Parisian ruin}

We now propose a new definition of actuarial ruin based on the occupation-time process of the surplus process $X$, namely the process
$$
t \mapsto \int_0^t \ind_{(-\infty,0)} (X_s) \mathrm{d}s .
$$
We define the time of cumulative Parisian ruin (at level $r>0$) by
\begin{equation}
\cumparisiantime = \inf \left\lbrace t > 0 \colon \int_0^t \ind_{(-\infty,0)} (X_s) \mathrm{d}s > r \right\rbrace .
\end{equation}

\begin{remark}
Note that for cumulative Parisian ruin, the parameter $r$ could be interpreted as the length of a clock started at the beginning of the first excursion, paused when the process returns above zero, and resumed at the beginning of the next excursion, and so on. In other words, it is the length of this single clock that is compared to the sum of all excursions below zero, as opposed to the previous definitions of Parisian ruin where a new clock is started at the beginning of each excursion.
\end{remark}

\begin{remark}
This definition of ruin is referred to as cumulative Parisian ruin due to its ties with cumulative Parisian options; see e.g.\ \cite{hugonnier_1999}. Note also that this definition of bankruptcy has been used as a definition of default in structural credit risk models; see \cites{yildirim_2006,makarov_et_al_2015}.
\end{remark}

We are interested in the following new ruin-based risk measure
$$
\p_x \left( \cumparisiantime \leq t \right) ,
$$
where $t>0$ is a fixed deterministic time. Clearly, for fixed values $t,r>0$, we have
$$
\left\lbrace \cumparisiantime > t \right\rbrace = \left\lbrace \int_0^t \ind_{(-\infty,0)} (X_s) \mathrm{d}s \leq r \right\rbrace ,
$$
so, the finite-time probability of cumulative Parisian ruin is given by
\begin{align}
\p_x \left( \cumparisiantime \leq t \right) &= 1 - \p_x \left( \int_0^t \ind_{(-\infty,0)} (X_s) \mathrm{d}s \leq r \right) \notag\\
&= 1 - \int_0^{r} \p_x \left( \int_0^t \ind_{(-\infty,0)} (X_s) \mathrm{d}s \in \mathrm{d}u \right) .\label{E:link_cumparisian_occ-time}
\end{align}
Then, we also have
\begin{align*}
\p_x \left( \cumparisiantime < \infty \right) &= 1 - \p_x \left( \int_0^\infty \ind_{(-\infty,0)} (X_s) \mathrm{d}s \leq r \right) \\
&= 1 - \int_0^r \p_x \left( \int_0^\infty \ind_{(-\infty,0)} (X_s) \mathrm{d}s \in \mathrm{d}u \right) .
\end{align*}


%

As the distribution of cumulative Parisian ruin is closely related to the distribution of the occupation-time process associated with the underlying surplus process $X$ over a finite-time horizon, we will now study the latter for the classical surplus processes. While there is a vast literature on the analysis of occupation times related to Brownian motion, it turns out that little is known about the distribution of the occupation-time process for a compound Poisson process with drift over a finite-time horizon.

Of course, it is very difficult to obtain analytical expressions for the probability of finite-time cumulative Parisian ruin for a general insurance risk process. However, there is a quite simple expression of its Laplace transform when the initial capital is zero in the case of spectrally negative Lévy processes.
\begin{prop}\label{P:occ-time0}
If $X$ is a SNLP such that $\e [X_1] = \psi'(0+) > 0$, then, for $p,q > 0$,
\begin{equation}\label{E:occ-time_0}
\int_0^\infty \mathrm{e}^{-p t} \e \left[ \mathrm{e}^{- q \int_0^{t} \ind_{(-\infty,0)} (X_s) \mathrm{d}s } \right] \mathrm{d}t = \frac{\Phi(p+q)}{p+q} \frac{1}{\Phi(p)} ,
\end{equation}
where $\Phi(q)$ is the largest root of Lundberg's equation $\psi(\theta)-q=0$, for a fixed $q>0$, and where $\psi(\theta) := \ln \left( \e \left[ \mathrm{e}^{\theta X_1} \right] \right)$.
\end{prop}
\begin{proof}
See e.g.\ Remark 4.1 in \cite{landriaultetal2014}.
\end{proof}

Since
$$
\int_0^\infty \mathrm{e}^{-p t} \e \left[ \mathrm{e}^{- q \int_0^{t} \ind_{(-\infty,0)} (X_s) \mathrm{d}s } \right] \mathrm{d}t = \int_0^{\infty} \int_0^{\infty} \ee^{-pt-qu} \p \left( \int_0^t \ind_{(-\infty,0)} (X_s) \mathrm{d}s \in \mathrm{d}u \right) \mathrm{d}t ,
$$
we obtain the distribution of the occupation time by inverting this double Laplace transform, which is our objective for the two main ruin models considered in this paper.
\subsection{Brownian ruin model}

One of Paul L\'evy's arcsine laws gives the distribution of the occupation time of the positive/negative half-line for a standard Brownian motion $B = \{B_t, t \geq 0\}$. More precisely, for $0<s<t$,
$$
\p \left( \int_0^t \ind_{(-\infty,0)} (B_u) \mathrm{d}u \leq s \right) = \frac{2}{\pi} \arcsin \left( \sqrt{\frac{s}{t}} \right) .
$$
This result was then extended to a Brownian motion with drift by Akahori \cite{akahori1995} and Tak\'acs \cite{takacs1996}.

Recall that in the Brownian ruin model of Equation~\eqref{E:Brownian}, the surplus is given by
$$
X^\text{BM}_t = X^\text{BM}_0 + \drift t + \sigma B_t ,
$$
where $\drift>0$ and $\sigma>0$. In this case, $\psi(\theta)=c\theta +{\sigma^2\over 2}\theta^2$, with right-inverse $\Phi(q)={1\over \sigma^2}\PAR{\sqrt{c^2+2\sigma^2 q}-c}$.

From Proposition~\ref{P:occ-time0}, the double Laplace transform of the cumulative Parisian ruin $\sigma_{r}$ is given by, for $p,q> 0$,
\begin{align*}
\int_0^\infty \mathrm{e}^{-p t} \e \left[ \mathrm{e}^{- q \int_0^{t} \ind_{(-\infty,0)} (X^\text{BM}_s) \mathrm{d}s } \right] \mathrm{d}t &= \frac{\sqrt{c^2+2\sigma^2 (p+q)}-c}{p+q} \frac{1}{\sqrt{c^2+2\sigma^2 p}-c}\\
&= \frac{1}{2\sigma^2p(p+q)} \PAR{\sqrt{c^2+2\sigma^2 (p+q)}-c}\PAR{\sqrt{c^2+2\sigma^2 p}+c}
\end{align*}

We recognize (1.4) in \cite{akahori1995} (given for the case $\sigma=1$) and then, by Laplace inversion, we recover the generalized arcsine law of the occupation time of the negative half-line for a Brownian motion with drift.

\begin{theorem}[Akahori (1995) \cite{akahori1995}]
For a fixed $t>0$, the distribution of the occupation time of the negative half-line, when $X^\text{BM}_0=0$, is given by
\begin{multline*}
\p \left( \int_0^t \ind_{(-\infty,0)} (X^\text{BM}_u) \mathrm{d}u \in \mathrm{d}s \right) = \frac{2}{\sigma^2} \left\lbrace \frac{\sigma \mathrm{e}^{-(c^2/2\sigma^2)s}}{\sqrt{2 \pi s}} - c \overline{N} \left( \frac{c \sqrt{s}}{\sigma} \right) \right\rbrace \\
\times \left\lbrace c + \frac{\sigma \mathrm{e}^{-(c^2/2\sigma^2)(t-s)}}{\sqrt{2 \pi(t-s)}} - c \overline{N} \left( \frac{c \sqrt{t-s}}{\sigma} \right) \right\rbrace \mathrm{d}s ,
\end{multline*}
for $0 \leq s \leq t$, where $\overline{N}$ denotes the tail of the standard normal distribution, that is
$$
\overline{N}(x) = \int_x^\infty \frac{\mathrm{e}^{-y^2/2}}{\sqrt{2 \pi}} \mathrm{d}y .
$$
\end{theorem}

Clearly, using Equation~\eqref{E:link_cumparisian_occ-time}, we have an explicit expression for $\p \left( \cumparisiantime \leq t \right)$ in the Brownian model. Using the continuity of Brownian motion, we can easily obtain an explicit expression for $\p_x \left( \cumparisiantime \leq t \right)$, i.e.\ when $X^\text{BM}_0=x>0$. Indeed, using the strong Markov property and the continuity of $X^\text{BM}$, we easily deduce that, for $s \leq t$,
\begin{multline*}
\p_x \left( \int_0^t \ind_{(-\infty,0)} (X^\text{BM}_u) \mathrm{d}u \leq s \right) \\
= \p_x \left( \tau_0^- > t-s \right) + \int_0^{t-s} \p_x \left( \tau_0^- \in \mathrm{d}u \right) \p \left( \int_0^{t-u} \ind_{(-\infty,0)} (X^\text{BM}_v) \mathrm{d}v \leq s \right) \mathrm{d}u ,
\end{multline*}
where it is well-known (see e.g.\ \cite{karatzasshreve1991}) that
$$
\p_x \left( \tau_0^- \in \mathrm{d}u \right) = \frac{x}{\sqrt{2 \pi \sigma^2 u^3}} \exp \left\lbrace - \frac{\left( x-cu \right)^2}{2 \sigma^2 u} \right\rbrace \mathrm{d}u .
$$
From the above, we easily deduce the following expression for the probability of cumulative Parisian ruin:
\begin{multline*}
\p_x \left( \cumparisiantime \leq t \right) = 1 - \p_x \left( \int_0^t \ind_{(-\infty,0)} (X^\text{BM}_u) \mathrm{d}u \leq r \right) \\
= \int_0^{t-r} \p \left( \int_0^{t-u} \ind_{(-\infty,0)} (X^\text{BM}_s) \mathrm{d}s > r \right) \frac{x}{\sqrt{2 \pi \sigma^2 u^3}} \exp \left\lbrace - \frac{\left( x-cu \right)^2}{2 \sigma^2 u} \right\rbrace \mathrm{d}u .
\end{multline*}

\subsection{Cram\'er-Lundberg model with exponential claims}

Recall that in the Cram\'er-Lundberg model with exponential claims introduced in Equation~\eqref{E:cramerlundberg}, the surplus is given by
$$
X^\text{CL}_t = X^\text{CL}_0 + \drift t - \sum_{i=1}^{N_t} C_i ,
$$
where $\drift>0$, $N = \{N_t, t \geq 0\}$ is a Poisson process with rate $\lambda>0$ and the $C_i$'s form a sequence of independent and identically distributed random variables following an exponential distribution with rate $\alpha>0$.

In this model, for $p\geq 0$, Lundberg's equation takes the form $\psi(\theta)-p=0$, where $\psi(\theta)= \drift \theta - {\lambda \theta \over \alpha+\theta}$, and its solutions are given by
\begin{align*}
\Phi(p)&={1\over2c}\PAR{p+\lambda-\drift \alpha+\sqrt{\Delta_p}} ,\\
\theta(p)&= {1\over2c}\PAR{p+\lambda-c\alpha -\sqrt{\Delta_p}} ,
\end{align*}
where $\Delta_p=(p+\lambda-c\alpha)^2+4c\alpha p$. Note that we can also write $\Delta_p=(p+\lambda+c\alpha)^2-4c\alpha \lambda$.

Clearly, for a fixed $t>0$, the distribution
\[
\p \left( \int_0^t \ind_{(-\infty,0)} (X^\text{CL}_u) \mathrm{d}u \in \mathrm{d}s \right)
\]
has a mass $a_t$ at $0$, which is the survival probability $\p \left( \tau_0^- > t \right)$. We now derive from Proposition~\ref{P:occ-time0} an expression for this distribution.
\begin{prop}\label{coro:Poisson}
For a fixed $t>0$, we have
\begin{equation}\label{E:occ-time_dist_cramer-lundberg}
\p \left( \int_0^t \ind_{(-\infty,0)} (X^\text{CL}_u) \mathrm{d}u \in \mathrm{d}s \right) = a_t \delta_0 \left( \mathrm{d}s \right) + a_{t-s} \left( \lambda - c \alpha (1-a_s) \right) \ind_{(0,t)} (s) \mathrm{d}s ,
\end{equation}
with
\begin{equation}\label{E:survival_proba_expclaims}
a_t= \left( 1-{\lambda\over c\alpha } \right)_+ + {2\lambda \over \pi} \ee^{-\PAR{\lambda+c\alpha}t}\int_{-1}^{1}{\sqrt{1-u^2}\over \lambda+c\alpha+2\sqrt{c\alpha \lambda}\, u} \ee^{-2\sqrt{c\alpha \lambda}\, t\, u} \mathrm{d}u,
\end{equation}
and $z_+:=\max(z,0)$.
%
\end{prop}
Note that there exists other equivalent expressions for $a_t = \p \left( \tau_0^- > t \right)$ in the Cram\'er-Lundberg model with exponential claims. See e.g.\ \cite{drekic2009}, or \cite[Proposition 1.3]{asmussen_albrecher_2010} and the references therein.

\begin{proof}
For simplicity, we will use $X$ instead of $X^\text{CL}$ in the proof. From Proposition~\ref{P:occ-time0}, we know that the double the Laplace transform of the occupation time of the negative half-line is equal to
\begin{align}\label{eq:LTPoisson}
\int_0^\infty \mathrm{e}^{-p t} \e \left[ \mathrm{e}^{- q \int_0^{t} \ind_{(-\infty,0)} (X_s) \mathrm{d}s } \right] \mathrm{d}t = \frac{\Phi(p+q)}{p+q} \frac{1}{\Phi(p)} .
\end{align}
Since $\Phi(v)/v$ converges to $1/c$ when $v$ goes to $\infty$, we deduce that~\eqref{eq:LTPoisson} converges to $\PAR{c\Phi(p)}^{-1}$ when $q$ goes to $\infty$. As mentioned above, this is due to the fact that $\int_0^t \ind_{(-\infty,0)} (X_u) \mathrm{d}u$ has a mass at $0$, which means that its distribution under $\p$ can be written in the following way:
\[
\p \left( \int_0^t \ind_{(-\infty,0)} (X_u) \mathrm{d}u \in \mathrm{d}s \right) = a_t \delta_0(\mathrm{d}s) + h_t(\mathrm{d}s) ,
\]
where $a_t \in [0,1]$ and $h_t(\mathrm{d}s)$ is a nonnegative measure on $[0,t]$ with no mass at zero.

We easily see that
\begin{equation}\label{E:laplace_transform_at}
\mathcal{L}a(p) := \int_0^\infty \mathrm{e}^{-p t} a_t \mathrm{d}t = \int_0^\infty \mathrm{e}^{-p t} \p \left( \tau_0^- > t \right) \mathrm{d}t = \frac{1}{c \Phi(p)} .
\end{equation}
Then, by linearity, the double Laplace transform of $(t,s) \mapsto h_t(\mathrm{d}s)$ must be equal to

\begin{align}
\mathcal{L}h(p,q)&={1\over c\Phi(p)}\PAR{{c\Phi(p+q)\over p+q}-1}\notag\\
&={1\over c\Phi(p)}\PAR{-{c\theta(p+q)\over p+q}+{\lambda-c\alpha\over p+q}}\notag\\
&={1\over c\Phi(p)}\PAR{{\alpha\over\Phi(p+q)}+{\lambda-c\alpha\over p+q}} ,\label{eq:TLh}
\end{align}
where the second equality is obtained using the explicit expressions for $\Phi(\cdot)$ and $\theta(\cdot)$, while the third equality is a consequence of the following identity: $c\Phi(p+q)\theta(p+q)=-\alpha (p+q)$.

Using basic properties of Laplace transforms and Equation~\eqref{E:laplace_transform_at}, the expression of the Laplace transform of $(t,s) \mapsto h_t(\mathrm{d}s)$ implies that
\[
h_t(\mathrm{d}s)=a_{t-s}(c\alpha a_{s}+\lambda-c\alpha) \ind_{(0,t)} (s) \mathrm{d}s.
\]
Consequently, to identify the law of $\int_0^t \ind_{(-\infty,0)} (X_u) \mathrm{d}u$, we need to identify $a_t$. Instead of using one of the available expressions from the literature, we derive a new expression.

First, using the conjugate expression of square roots and the explicit form of $\Phi(p)$, we notice that
\begin{align*}
\mathcal{L}a(p)
&={\sqrt{(p+\lambda+c\alpha)^2-4c\alpha \lambda}-(p+\lambda+c\alpha)\over 2c\alpha p}+{1\over p}\\
&={-2\lambda\over p\PAR{\sqrt{(p+\lambda+c\alpha)^2-4c\alpha \lambda}+(p+\lambda+c\alpha)}}+{1\over  p}.
\end{align*}

As, for functions $f$ and $g$, $\mathcal{L}f(p)={1\over p}\int_0^p\mathcal{L}g(q)\dd q$ implies $f(t)=\int_t^\infty s^{-1}g(s)\mathrm{d}s$ and the Laplace transform of $t\mapsto I_1(kt)$ is equal to $p\mapsto k\PAR{\sqrt{p^2-k^2}+p}^{-1}\PAR{p^2-k^2}^{-1/2}$ with  $I_1$ the modified Bessel function of the first kind of order $1$ (see e.g. \cite[Supplement 4]{Polyetal2008} for more details), we deduce that, when $\lambda \geq c\alpha$, or equivalently $\e[X_1] \leq 0$,
\[
a_t=\sqrt{\lambda \over   c\alpha}\int_t^{\infty}{1\over s}\ee^{-\PAR{\lambda+c\alpha}s}I_1(2\sqrt{c\alpha \lambda}\, s) \mathrm{d}s ,
\]
and, when $\lambda< c\alpha$, or equivalently $\e[X_1] > 0$, 
\[
a_t=1-{\lambda \over c\alpha }+\sqrt{\lambda \over   c\alpha}\int_t^{\infty}{1\over s}\ee^{-\PAR{\lambda+c\alpha}s}I_1(2\sqrt{c\alpha \lambda}\, s) \mathrm{d}s .
\]

At last, the final result is obtained using Fubini's theorem and the following integral form for the modified Bessel function: for $s>0$,
\[
I_1(s) = {s\over \pi} \int_{-1}^1 \mathrm{e}^{-su} \sqrt{1-u^2} \mathrm{d}u .
\]
\end{proof}

As opposed to the Brownian case, it is more delicate to obtain the corresponding distribution under $\p_x$ for $x>0$, i.e.\ when $X^\text{CL}_0=x>0$, since $X_{\tau_0^-}<0$ a.s..
%
%

Again note that, for a fixed $t>0$, the distribution
\[
\p_x \left( \int_0^t \ind_{(-\infty,0)} (X^\text{CL}_u) \mathrm{d}u \in \mathrm{d}s \right)
\]
has a mass $a_t^x$ at $0$, which is the survival probability $\p_x \left( \tau_0^- > t \right)$.
%
\begin{theorem}\label{main_result}
For a fixed $t>0$, we have
\begin{equation}\label{E:occ-time_dist_cramer-lundberg_x}
\p_x \left( \int_0^t \ind_{(-\infty,0)} (X^\text{CL}_u) \mathrm{d}u \in \mathrm{d}s \right) = a^x_t \delta_0 (\mathrm{d}s) + \left( a_{t-s}^x+k^x_{t-s} \right) \left( \lambda - c \alpha \left( 1 - a^0_{s} \right) \right) \ind_{(0,t)} (s) \mathrm{d}s ,
\end{equation}
with
\begin{align}
a^x_t &= 1-\lambda \mathrm{e}^{-\alpha x} \int_0^t\mathrm{e}^{-(\lambda+c\alpha)s} \SBRA{I_0\PAR{2\sqrt{\lambda c\alpha} \sqrt{s\PAR{s+x/c}}}-{s\over s+x/c}I_2\PAR{2\sqrt{\lambda c\alpha}\sqrt{s\PAR{s+x/c}}}} \mathrm{d}s , \notag \\
k^x_t &= \mathrm{e}^{-\alpha x} - 1 \label{E:kxt} \\
& \qquad + x \alpha \lambda \mathrm{e}^{-\alpha x} \int_0^t s^{-1} \mathrm{e}^{-(\lambda+c\alpha)s}\SBRA{I_0\PAR{2\sqrt{c\alpha \lambda}\sqrt{s(s+x/c)}}-I_2\PAR{2\sqrt{c\alpha \lambda}\sqrt{s(s+x/c)}}}\dd s , \notag
\end{align}
where $I_\nu$ represents the modified Bessel function of the first kind of order $\nu$, which can be written as: for $s\geq 0$,
\[
I_0(s)={1\over \pi}\int_{-1}^1\mathrm{e}^{-su}\PAR{1-u^2}^{- 1/2} \mathrm{d}u \quad \text{and} \quad
I_2(s)={s^2\over 3\pi}\int_{-1}^1\mathrm{e}^{-su} \PAR{1-u^2}^{3/2} \mathrm{d}u .
\]
\end{theorem}

\begin{proof}
For simplicity, we will use $X$ instead of $X^\text{CL}$ in the proof. First, note that we have the following relationships between $\Phi(\cdot)$ and $\theta(\cdot)$: for $p,q \geq 0$, 
\begin{gather}
\Phi(p)-\theta(p)={1\over c}\sqrt{\Delta_p} ,\label{eq:diff} \\
\Phi(p)\theta(p)=-{\alpha p\over c}, \\
\PAR{\theta(p+q)-\Phi(p)}\PAR{\theta(p+q)-\theta(p)}={q\over c}\PAR{\theta(p+q)+\alpha} . \label{eq:prodTheta}
\end{gather}

From \cite[Corollary 1]{guerin_renaud_2015}, for $x \in \mathbb{R}$ and $p,q>0$, we deduce an explicit expression for the double Laplace transform of the occupation time of the negative half-line
\begin{multline*}
\int_0^\infty \mathrm{e}^{-p t} \e_x \left[ \mathrm{e}^{- q \int_0^{t} \ind_{(-\infty,0)} (X_s) \mathrm{d}s } \right] \mathrm{d}t \\
=\int_{-\infty}^{\infty}\BRA{\left( \frac{\Phi(p+q)-\Phi(p)}{q} \right) \cH^\PAR{p+q,-q}(x) \cH^\PAR{p,q}(-y) - \mathcal W_{x}^{(p,q)} (x-y)} \dy ,
\end{multline*}
where, in the case of a Cram\'er-Lundberg risk process with exponential claims (see e.g.\ \cite{guerin_renaud_2015} for more details), for $x\geq 0$,
$$
\mathcal{H}^\PAR{p,q}(x) = {q\over \sqrt{\Delta_{p+q}}}\SBRA{{\alpha+\Phi(p+q)\over \Phi(p+q)-\Phi(p)}\mathrm{e}^{\Phi(p+q)x}-{\alpha+\theta(p+q)\over \theta(p+q)-\Phi(p)}\mathrm{e}^{\theta(p+q) x}}
$$
and
\[
\cH^\PAR{p+q,-q}(x)=
{q\over \sqrt{\Delta_{p}}}\PAR{{\alpha+\Phi(p)\over \Phi(p+q)-\Phi(p)}\mathrm{e}^{\Phi(p)x}-{\alpha+\theta(p)\over \Phi(p+q)-\theta(p)}\mathrm{e}^{\theta(p) x}} ,
\]
and where, for $x<0$, $\mathcal{H}^\PAR{p,q}(x) = \exp\PAR{\Phi(p) x}$ and $\mathcal{H}^\PAR{p+q,-q}(x) = \exp\PAR{\Phi(p+q) x}$. Also, for $x\geq a$, 
\begin{align*}
\wapq(x) &= q {\alpha+\Phi(p+q)\over \sqrt{\Delta_{p+q}\Delta_p}}\ee^{\Phi(p+q)(x-a)}\SBRA{{\alpha+\Phi(p)\over \Phi(p+q)-\Phi(p)}\ee^{\Phi(p)a}-{\alpha+\theta(p)\over \Phi(p+q)-\theta(p)}\ee^{\theta(p)a}}\\
&\hskip 0.5cm -q{\alpha+\theta(p+q)\over \sqrt{\Delta_{p+q}\Delta_p}}\ee^{\theta(p+q)(x-a)}\SBRA{{\alpha+\Phi(p)\over \theta(p+q)-\Phi(p)}\ee^{\Phi(p)a}-{\alpha+\theta(p)\over \theta(p+q)-\theta(p)}\ee^{\theta(p)a}} ,
\end{align*}
for $x \in [0,a]$, $\wapq(x) = \pscale(a)$, and finally, for $x<0$, $\wapq(x) =  0$. Therefore, for $y<0$, we have
\begin{multline*}
\mathcal W_{x}^{(p,q)} (x-y)
 ={\alpha+\Phi(p+q)\over \sqrt{\Delta_{p+q}}}\cH^\PAR{p+q,-q}(x) \mathrm{e}^{-\Phi(p+q)y}\\
 -q{\alpha+\theta(p+q)\over \sqrt{\Delta_{p+q}\Delta_p}}\SBRA{{\alpha+\Phi(p)\over \theta(p+q)-\Phi(p)}\ee^{\Phi(p)x}-{\alpha+\theta(p)\over \theta(p+q)-\theta(p)}\ee^{\theta(p)x}}\ee^{-\theta(p+q)y}.
\end{multline*}

We divide the computation of the integral into three parts. We first consider the integral on the negative half-line. For $y<0$, using~\eqref{eq:prodTheta} and~\eqref{eq:diff}, we have
\begin{align*}
&\left( \frac{\Phi(p+q)-\Phi(p)}{q} \right) \cH^\PAR{p+q,-q}(x) \cH^\PAR{p,q}(-y) - \mathcal W_{x}^{(p,q)} (x-y) 
\\
 &={\alpha+\theta(p+q)\over \sqrt{\Delta_{p+q}}}\left[-{\Phi(p+q)-\Phi(p)\over \theta(p+q)-\Phi(p)}\cH^\PAR{p+q,-q}(x) \right.\\
&\hskip 3cm \left.+{q\over \sqrt{\Delta_p}}\PAR{{\alpha+\Phi(p)\over \theta(p+q)-\Phi(p)}\ee^{\Phi(p)x}-{\alpha+\theta(p)\over \theta(p+q)-\theta(p)}\ee^{\theta(p)x}}\right]\mathrm{e}^{-\theta(p+q) y} \\
&={c\PAR{\alpha+\theta(p)}\over \sqrt{\Delta_p\Delta_{p+q}}}\left[\PAR{\Phi(p+q)-\Phi(p)}{\PAR{\theta(p+q)-\theta(p)}\over \Phi(p+q)-\theta(p)}-\PAR{\theta(p+q)-\Phi(p)}\right]\mathrm{e}^{\theta(p) x}\mathrm{e}^{-\theta(p+q) y}  \\
&={\alpha+\theta(p)\over c\PAR{ \Phi(p+q)-\theta(p)}}\mathrm{e}^{\theta(p) x} \mathrm{e}^{-\theta(p+q) y}.
\end{align*}
Consequently,
\begin{multline}\label{eq:intneg}
\int_{-\infty}^{0}\BRA{\left( \frac{\Phi(p+q)-\Phi(p)}{q} \right) \cH^\PAR{p+q,-q}(x) \cH^\PAR{p,q}(-y) - \mathcal W_{x}^{(p,q)} (x-y)} \dy\\
=-{\alpha+\theta(p)\over c\theta(p+q)\PAR{ \Phi(p+q)-\theta(p)}}\mathrm{e}^{\theta(p) x}.
\end{multline}
We now compute the integral on the positive half-line for each term. We have
\begin{multline}\label{eq:intposH}
\int_{0}^{\infty}\left( \frac{\Phi(p+q)-\Phi(p)}{q} \right) \cH^\PAR{p+q,-q}(x) \cH^\PAR{p,q}(-y) \dy\\
=\frac{\Phi(p+q)-\Phi(p)}{q}  \cH^\PAR{p+q,-q}(x)\int_{0}^{\infty} \mathrm{e}^{-\Phi(p) y} \dy
= \frac{\Phi(p+q)-\Phi(p)}{q\Phi(p)}  \cH^\PAR{p+q,-q}(x).
\end{multline}
At last, we notice that $x-y<x$ when $y>0$ and then $\mathcal W_x^{(p,q)}(x-y)=\pscale(x-y)$. Consequently,
\begin{equation}\label{eq:intposW}
-\int_{0}^{\infty}\mathcal  W_{x}^{(p,q)}(y) \dy
=-\int_{0}^{x}\pscale (y) \dy=-{\alpha+\Phi(p)\over\Phi(p) \sqrt{\Delta_p}}\mathrm{e}^{\Phi(p)x}+{\alpha+\theta(p)\over \theta(p)\sqrt{\Delta_p}}\mathrm{e}^{\theta(p) x}-{\alpha\over c\Phi(p)  \theta(p)}.
 \end{equation}
At last, summing up the three integrals in~\eqref{eq:intneg}, \eqref{eq:intposH} and \eqref{eq:intposW}, we obtain
\begin{multline}\label{eq:LTruine}
\int_0^\infty \mathrm{e}^{-p t} \e_x \left[ \mathrm{e}^{- q \int_0^{t} \ind_{(-\infty,0)} (X_s) \mathrm{d}s } \right] \mathrm{d}t =\\
{\PAR{\alpha+\theta(p)}\Phi(p+q)\over c\PAR{\Phi(p+q)-\theta(p)}}\SBRA{
{1\over  \Phi(p)\theta(p)}-{1\over \Phi(p+q)\theta(p+q)}}\mathrm{e}^{\theta(p) x}
-{\alpha\over c\Phi(p)  \theta(p)}.
 \end{multline}
 
The rest of the proof is similar to the proof of Proposition~\ref{coro:Poisson}, i.e.\ we want to invert the double Laplace transform in~\eqref{eq:LTruine}. Fix $p>0$. Since when $q\rightarrow \infty$ we have $\Phi(p+q)\rightarrow \infty$ and $\theta(p+q)\rightarrow -\alpha$, then the previous quantity converges to
\begin{equation}\label{eq:LTa}
{\alpha+\theta(p)\over c  \Phi(p)\theta(p)}
\mathrm{e}^{\theta(p) x}
-{\alpha\over c\Phi(p)  \theta(p)}=-{\alpha+\theta(p)\over \alpha p}\mathrm{e}^{\theta(p) x}+{1\over p}.
\end{equation}
Since the distribution of $ \int_0^t \ind_{(-\infty,0)} (X_u) \mathrm{d}u$ has a mass at $0$,
\[
\p_x \left( \int_0^t \ind_{(-\infty,0)} (X_u) \mathrm{d}u \in \mathrm{d}s \right) = a^x_t\delta_0(\mathrm{d}s)+h^x_t(\mathrm{d}s) ,
\]
where $a^x_t=\dP_x(\tau_0^->t)\in[0,1]$ and $h^x_t(\mathrm{d}s)$ is a nonnegative measure on $[0,t]$ with no mass at zero. From \eqref{eq:LTa}, we deduce that the Laplace transform of $t\mapsto a^x_t$ is equal to ${1\over p}-{\alpha+\theta(p)\over \alpha p}
\mathrm{e}^{\theta(p) x}
$ 
and then, from~\eqref{eq:LTruine}, the Laplace transform of $(s,t)\mapsto h^x_t(\dd s)$ is
\[
\mathcal{L}h^x (p,q)={\PAR{\alpha+\theta(p)}\Phi(p+q)\over c\PAR{\Phi(p+q)-\theta(p)}}\SBRA{
{1\over  \Phi(p)\theta(p)}-{1\over \Phi(p+q)\theta(p+q)}}\mathrm{e}^{\theta(p) x}
-{\alpha+\theta(p)\over c  \Phi(p)\theta(p)}\mathrm{e}^{\theta(p) x}.
\]
Taking $x=0$ in~\eqref{eq:LTruine} and comparing with~\eqref{eq:LTPoisson}, we deduce that  the Laplace transform of $(t,s)\mapsto h^x_t(\mathrm{d} s)$ is equal to
\begin{equation}\label{eq:h}
\mathcal{L}h^x (p,q)
=\PAR{{c   \Phi(p+q)\over p+q}-1}\,{\mathrm{e}^{\theta(p) x}\over c \Phi(p)}.
\end{equation}
Therefore, to identify the law of $\int_0^t \ind_{(-\infty,0)} (X_u) \mathrm{d}u$, we need to identify
$$
a^x_t = \p_x \left( \tau_0^- > t \right) = 1 - \int_0^t \p_x \left( \tau_0^- \in \mathrm{d}s \right) .
$$

Using the explicit expression for $\p_x \left( \tau_0^- \in \mathrm{d}s \right)$ given in \cite[Equation (D.17)]{drekic2009},  we deduce that
\[
a^x_t=1-\lambda \mathrm{e}^{-\alpha x}\int_0^t\mathrm{e}^{-(\lambda+c\alpha)s}\SBRA{I_0\PAR{2\sqrt{\lambda c\alpha}\sqrt{s\PAR{s+x/c}}}-{s\over s+x/c}I_2\PAR{2\sqrt{\lambda c\alpha}\sqrt{s\PAR{s+x/c}}}} \mathrm{d}s ,
\] 
where $I_\nu$ represents the modified Bessel function of the first kind of order $\nu$, which can be written as
\[
I_0(s)={1\over \pi}\int_{-1}^1\mathrm{e}^{-su}\PAR{1-u^2}^{-1/2} \mathrm{d}u \quad \text{and} \quad I_2(s)={s^2\over 3\pi}\int_{-1}^1\mathrm{e}^{-su} \PAR{1-u^2}^{3/2} \mathrm{d}u.
\]
Moreover, using~\eqref{eq:diff}, the Laplace transform of $h^x$ can be written in terms of the Laplace transform of $a^x$:
%
%
%
%
\[\mathcal{L}h^x(p,q)=\SBRA{{c   \Phi(p+q)\over p+q}-1}\, \SBRA{\mathcal{L}a^x(p)+{1\over p}\PAR{\mathrm{e}^{\theta(p)x}-1}}.
\] 
We then deduce that (from the proof of Proposition~\ref{coro:Poisson})
\[
h^x_t(\mathrm{d} s) = c \alpha \left( a^0_{s} + {\lambda \over c\alpha}-1 \right) \left( a_{t-s}^x+k^x_{t-s} \right) \ind_{(0,t)} (s) \mathrm{d}s ,
\]
with $\mathcal{L}k^x(p)={1\over p}\PAR{\mathrm{e}^{\theta(p)x} -1}$. As, for functions $f$ and $g$, $\mathcal{L}f(p)={1\over p}\int_p^\infty\mathcal{L}g(q)\dd q$ implies $f(t)=\int_0^t s^{-1}g(s)\mathrm{d}s$,  we have $k^x_t=\int_0^t s^{-1} m^x_s\dd s + \mathrm{e}^{-\alpha x}-1 $, with $\mathcal{L}m^x(p) = -x \theta^\prime (p) \mathrm{e}^{\theta(p)x}$.
Let us now explicit the function $m^x_t$. We have
\begin{align*}
\mathcal{L}m^x (p) &= -{x\over 2c}\PAR{1-{p+\lambda+c\alpha \over \sqrt{(p+\lambda+c\alpha)^2-4c\alpha \lambda}}}\mathrm{e}^{{x\over 2c}\PAR{p+\lambda-c\alpha-\sqrt{(p+\lambda+c\alpha)^2-4c\alpha \lambda}}}\\
&=-{x\mathrm{e}^{-x\alpha}\over 2c}\PAR{ \sqrt{(p+\lambda+c\alpha)^2-4c\alpha \lambda}-\PAR{p+\lambda+c\alpha}} {\mathrm{e}^{-{x\over 2c}\PAR{\sqrt{(p+\lambda+c\alpha)^2-4c\alpha \lambda}-(p+\lambda+c\alpha)}}\over \sqrt{(p+\lambda+c\alpha)^2-4c\alpha \lambda}} .
\end{align*}
Using similar inversion techniques as in \cite[(D.13)-(D.15)]{drekic2009}, the inverse Laplace transform of the previous expression can be expressed in terms of the modified Bessel function $I_1$:
\[
m^x_t  = {x\mathrm{e}^{-x\alpha}\over 2c}\mathrm{e}^{-(\lambda+c\alpha)t}{2\sqrt{c\alpha \lambda}\over \sqrt{t(t+x/c)}}I_1\PAR{2\sqrt{c\alpha \lambda}\sqrt{t(t+x/c)}}.
\]
Since $I_1(z)=z\PAR{I_0(z)-I_2(z)}$, we have
\[
m^x_t  = x\alpha \lambda \mathrm{e}^{-x\alpha}\mathrm{e}^{-(\lambda+c\alpha)t}\SBRA{I_0\PAR{2\sqrt{c\alpha \lambda}\sqrt{t(t+x/c)}}-I_2\PAR{2\sqrt{c\alpha \lambda}\sqrt{t(t+x/c)}}},
\]
and then
\begin{multline*}
k^x_t = \mathrm{e}^{-\alpha x} - 1 \label{E:kxt} \\
+ x \alpha \lambda \mathrm{e}^{-\alpha x} \int_0^t s^{-1} \mathrm{e}^{-(\lambda+c\alpha)s}\SBRA{I_0\PAR{2\sqrt{c\alpha \lambda}\sqrt{s(s+x/c)}}-I_2\PAR{2\sqrt{c\alpha \lambda}\sqrt{s(s+x/c)}}}\dd s .
\end{multline*}

\end{proof}

If we set $x=0$ in the definition of $k^x_t$ in~\eqref{E:kxt}, then clearly $k_t^0=0$. Therefore, Theorem~\ref{main_result} is a generalization of Proposition~\ref{coro:Poisson}.

\subsection{Relationships with other definitions of ruin}

When $r \to 0$, we should intuitively recover the classical probability of finite-time ruin. Indeed, when $r$ is very small, then almost no grace period is granted when the surplus goes below zero.

First, for $x \geq 0$ and $t>0$, we have for every $r>0$ that
$$
\left\lbrace \int_0^t \ind_{(-\infty,0)} (X_s) \mathrm{d}s > r \right\rbrace \subset \left\lbrace \tau_0^- \leq t \right\rbrace
$$
so
$$
\p_x \left( \cumparisiantime \leq t \right) \leq \p_x \left( \tau_0^- \leq t \right) .
$$
Therefore, as a risk measure, the probability of cumulative Parisian ruin is less \textit{conservative} than the probability of ruin.

We can obtain a much stronger relationship between cumulative Parisian ruin and classical ruin in the full generality of a L\'evy insurance risk model. Recall that a L\'evy insurance risk process is a process $X = \{X_t, t \geq 0\}$ with stationary and independent increments and no positive jumps.
%
%

\begin{prop}
Assume $X$ is a L\'evy insurance risk process. For all $x\geq 0$, the time of cumulative Parisian ruin $\cumparisiantime$ converges $\p_x$-almost surely, as $r \to 0$, to the time of classical ruin $\tau_0^-$.
\end{prop}
\begin{proof}
First, we show convergence in probability (with respect to $\p_x$, for an arbitrary value of $x\geq0$), that is, for every $\epsilon>0$,
$$
\p_x \left( \left\vert \cumparisiantime - \tau_0^- \right\vert > \epsilon \right) \underset{r \to 0}{\longrightarrow} 0.
$$
Indeed, for any $\varepsilon>0$, we note that $\sigma_r\geq \tau_0^-$ and, by the strong Markov property,
\begin{equation*}
\dP_x\PAR{\ABS{\cumparisiantime-\tau_0^-}>\epsilon}=\dP_x\PAR{\int_{\tau_0^-}^{\epsilon+\tau_0^-} \ind_{(-\infty,0)} (X_s) \mathrm{d}s\leq r}=\dE_x\SBRA{\dP_{X_{\tau_0^-}}\PAR{\int_0^\varepsilon \ind_{(-\infty,0)} (X_s) \mathrm{d}s\leq r}} .
\end{equation*}
Therefore we want to show that $\dE_x\SBRA{\dP_{X_{\tau_0^-}}\PAR{\int_0^\varepsilon \ind_{(-\infty,0)} (X_s) \mathrm{d}s = 0}}=0$. Clearly, on $\{X_{\tau_0^-}<0\}$, since the process is skip-free upward,
$$
\dP_{X_{\tau_0^-}} \PAR{\int_0^\varepsilon \ind_{(-\infty,0)} (X_s) \mathrm{d}s = 0} = 0 .
$$
On the other hand, if $x>0$, it is well known that $\p_x \left( X_{\tau_0^-}=0 \right)$ if and only if $X$ has a Brownian component. Finally, recall that a SNLP has paths of unbounded variation (e.g.\ with a Brownian component) if and only if $0$ is regular for $(-\infty,0)$, which means that when the process sits at zero it will enter $(-\infty,0)$ immediately (and will start to accumulate occupation time). Putting the pieces together concludes the proof of convergence in probability.
%

Since $r \mapsto \cumparisiantime$ is a non-decreasing function, we deduce that $\cumparisiantime$ converges to $\tau_0^-$ a.s.\ when $r$ goes to $0$.
\end{proof}


At first sight, it is very tempting to think of exponential Parisian ruin as being \textit{one Laplace transform away} from deterministic Parisian ruin. As discussed in \cite{loeffenetal2013}, due to the choice of different exponential clocks for different excursions below zero, this is not the case. In fact, it turns out that cumulative Parisian ruin is the notion of Parisian ruin that can be connected to exponential Parisian ruin through a Laplace transform, as the next result will show.

Note that the next result holds for any L\'evy insurance risk process, including of course the Cram\'er-Lundberg risk process and the Brownian risk process studied previously.
\begin{prop}
Assume $X$ is a L\'evy insurance risk process and let $\mathbf{e}_q$ be an independent exponential random variable with rate $q>0$. Then, for all initial surplus $x\geq 0$, $\kappa_q$ and $\sigma_{\mathbf{e}_q}$ are such that, for all $x\geq 0$, $t>0$,
$$
\p_x \left( \parisiantime \leq t \right) = \p_x \left( \sigma_{\mathbf{e}_q} \leq t \right)
$$
and, moreover,
$$
\p_x \left( \parisiantime < \infty \right) = \p_x \left( \sigma_{\mathbf{e}_q} < \infty \right) .
$$
\end{prop}

\begin{proof}
Comparing \cite[Corollary 1]{guerin_renaud_2015} and \cite[Corollary 1.1]{baurdoux_et_al_2015}, we observe that,
%
%
for all $y \in (-\infty,0]$,
$$
q \int_0^\infty \mathrm{e}^{-\theta t} \e_x \left[ \mathrm{e}^{- q \int_0^{t} \ind_{(-\infty,0)} (X_s) \mathrm{d}s } ; X_t \in \mathrm{d}y \right] \mathrm{d}t = \e_x \left[ \mathrm{e}^{-\theta \parisiantime}, X_{\parisiantime} \in \mathrm{d}y , \parisiantime < \infty \right] .
$$
Using Fubini's theorem and then integrating over $(-\infty,0]$ (with respect to $y$) both distributions, we get
$$
q \int_0^\infty \mathrm{e}^{-\theta t} \e_x \left[ \mathrm{e}^{- q \int_0^{t} \ind_{(-\infty,0)} (X_s) \mathrm{d}s } ; X_t < 0 \right] \mathrm{d}t = \e_x \left[ \mathrm{e}^{-\theta \parisiantime}, \parisiantime < \infty \right] .
$$
Note that we can write
\begin{align*}
\e_x \left[ \mathrm{e}^{-\theta \parisiantime}, \parisiantime < \infty \right]
&= \int_0^\infty \mathrm{e}^{-\theta t} \p_x \left( \parisiantime \in \mathrm{d}t \right) .
\end{align*}
Therefore, by Laplace inversion, we have, on $(0,\infty)$,
$$
q \e_x \left[ \mathrm{e}^{- q \int_0^{t} \ind_{(-\infty,0)} (X_s) \mathrm{d}s } ; X_t < 0 \right] \mathrm{d}t = \p_x \left( \parisiantime \in \mathrm{d}t \right) .
$$

Integrating over $(0,t)$ for an arbitrary value of $t$ yields
\begin{align*}
\p_x \left( \parisiantime \leq t \right) &= \int_0^t q \e_x \left[ \mathrm{e}^{- q \int_0^{s} \ind_{(-\infty,0)} (X_u) \mathrm{d}u } ; X_s < 0 \right] \mathrm{d}s \\
&= \e_x \left[ \int_0^t q \ind_{(-\infty,0)} (X_s) \mathrm{e}^{- q \int_0^{s} \ind_{(-\infty,0)} (X_u) \mathrm{d}u } \mathrm{d}s \right] \\
&= \e_x \left[ 1 - \mathrm{e}^{- q \int_0^{t} \ind_{(-\infty,0)} (X_s) \mathrm{d}s } \right] \\
&= 1 - \e_x \left[ \mathrm{e}^{- q \int_0^{t} \ind_{(-\infty,0)} (X_s) \mathrm{d}s } \right] .
\end{align*}
So, we have
\begin{equation}\label{E:probaexpparisianruinfinite}
\p_x \left( \parisiantime \leq t \right) = 1 - \p_x \left( \int_0^{t} \ind_{(-\infty,0)} (X_s) \mathrm{d}s \leq \mathbf{e}_q \right) = \p_x \left( \sigma_{\mathbf{e}_q} \leq t \right) ,
\end{equation}
and, in particular,
\begin{equation}\label{E:probaexpparisianruin}
\p_x \left( \parisiantime < \infty \right) = 1 - \e_x \left[ \mathrm{e}^{- q \int_0^{\infty} \ind_{(-\infty,0)} (X_s) \mathrm{d}s } \right] = \p_x \left( \sigma_{\mathbf{e}_q} < \infty \right) .
\end{equation}
\end{proof}

\begin{remark}
The previous proof highlights a strong relationship between the distribution of the occupation-time process and exponential Parisian ruin, namely identity~\eqref{E:probaexpparisianruinfinite}. The equality in Equation~\eqref{E:probaexpparisianruin} was first discussed in \cite{landriaultetal2011}; note that Equation~\eqref{E:probaexpparisianruinfinite} is a finite-time version of it.
\end{remark}

\section{Acknowledgements}

J.-F. Renaud was visiting IRMAR at Universit\'e de Rennes 1 when part of this work was carried out and he is grateful for their hospitality.

Funding in support of this work was provided by the Natural Sciences and Engineering Research Council of Canada (NSERC).

\bibliographystyle{abbrv}
\bibliography{../occupation}

\begin{bibdiv}
\begin{biblist}

\bib{akahori1995}{article}{
      author={Akahori, J.},
       title={Some formulae for a new type of path-dependent option},
        date={1995},
     journal={Ann. Appl. Probab.},
      volume={5},
      number={2},
       pages={383\ndash 388},
}

\bib{albrecheretal2011a}{article}{
      author={Albrecher, H.},
      author={Cheung, E. C.~K.},
      author={Thonhauser, S.},
       title={Randomized observation periods for the compound {P}oisson risk
  model: dividends},
        date={2011},
     journal={Astin Bull.},
      volume={41},
      number={2},
       pages={645\ndash 672},
}

\bib{albrecheretal2013}{article}{
      author={Albrecher, H.},
      author={Cheung, E. C.~K.},
      author={Thonhauser, S.},
       title={Randomized observation periods for the compound {P}oisson risk
  model: the discounted penalty function},
        date={2013},
     journal={Scand. Actuar. J.},
      number={6},
       pages={424\ndash 452},
}

\bib{albrecheretal2011b}{article}{
      author={Albrecher, H.},
      author={Gerber, H.~U.},
      author={Shiu, E. S.~W.},
       title={The optimal dividend barrier in the {G}amma-{O}mega model},
        date={2011},
     journal={Eur. Actuar. J.},
      volume={1},
      number={1},
       pages={43\ndash 55},
}

\bib{albrecher_lautscham_2013}{article}{
      author={Albrecher, H.},
      author={Lautscham, V.},
       title={From ruin to bankruptcy for compound {P}oisson surplus
  processes},
        date={2013},
     journal={Astin Bull.},
      volume={43},
      number={2},
       pages={213\ndash 243},
}

\bib{asmussen_albrecher_2010}{book}{
      author={Asmussen, S.},
      author={Albrecher, H.},
       title={Ruin probabilities},
     edition={Second},
   publisher={World Scientific},
        date={2010},
}

\bib{baurdoux_et_al_2015}{article}{
      author={Baurdoux, E.~J.},
      author={Pardo, J.~C.},
      author={P\'erez, J.~L.},
      author={Renaud, J.-F.},
       title={Gerber-{S}hiu distribution at {P}arisian ruin for {L}\'evy
  insurance risk processes},
        date={to appear},
     journal={J. Appl. Probab.},
}

\bib{czarnapalmowski2010}{article}{
      author={Czarna, I.},
      author={Palmowski, Z.},
       title={Ruin probability with {P}arisian delay for a spectrally negative
  {L}\'evy risk process},
        date={2011},
     journal={J. Appl. Probab.},
      volume={48},
      number={4},
       pages={984\ndash 1002},
}

\bib{dosreis1993}{article}{
      author={dos Reis, A.~E.},
       title={How long is the surplus below zero?},
        date={1993},
     journal={Insurance Math. Econom.},
      volume={12},
      number={1},
       pages={23\ndash 38},
}

\bib{drekic2009}{article}{
      author={Drekic, S.},
       title={Discussion of "{O}n the joint distributions of the time to ruin,
  the surplus prior to ruin, and the deficit at ruin in the classical risk
  model, by {D}. {L}andriault and {G}.{E}. {W}illmot, vol. 13, no. 2, 2009"},
        date={2009},
     journal={N. Am. Actuar. J.},
      volume={13},
      number={3},
       pages={404\ndash 406},
}

\bib{gerbershiu1998}{article}{
      author={Gerber, H.~U.},
      author={Shiu, E. S.~W.},
       title={On the time value of ruin},
        date={1998},
     journal={N. Am. Actuar. J.},
      volume={2},
      number={1},
       pages={48\ndash 78},
}

\bib{gerber_et_al_2012}{article}{
      author={Gerber, H.~U.},
      author={Shiu, E. S.~W.},
      author={Yang, H.},
       title={The {O}mega model: from bankruptcy to occupation times in the
  red},
        date={2012},
     journal={Eur. Actuar. J.},
      volume={2},
      number={2},
       pages={259\ndash 272},
}

\bib{guerin_renaud_2015}{article}{
      author={Gu\'erin, H.},
      author={Renaud, J.-F.},
       title={Joint distribution of a spectrally negative {L}\'evy process and
  its occupation time, with step option pricing in view},
        date={to appear},
     journal={Adv. in Appl. Probab.},
}

\bib{hugonnier_1999}{article}{
      author={Hugonnier, J.-N.},
       title={The {F}eynman-{K}ac formula and pricing occupation time
  derivatives},
        date={1999},
     journal={Int. J. Theor. Appl. Finance},
      volume={2},
      number={2},
       pages={153\ndash 178},
}

\bib{karatzasshreve1991}{book}{
      author={Karatzas, I.},
      author={Shreve, S.~E.},
       title={Brownian motion and stochastic calculus},
     edition={Second},
   publisher={Springer-Verlag},
     address={New York},
        date={1991},
}

\bib{kyprianou2013}{book}{
      author={Kyprianou, A.~E.},
       title={Gerber-{S}hiu risk theory},
      series={EAA Series},
   publisher={Springer},
        date={2013},
}

\bib{kyprianou2014}{book}{
      author={Kyprianou, A.~E.},
       title={Fluctuations of {L}\'evy processes with applications -
  {I}ntroductory lectures},
     edition={Second},
      series={Universitext},
   publisher={Springer, Heidelberg},
        date={2014},
}

\bib{landriaultetal2011}{article}{
      author={Landriault, D.},
      author={Renaud, J.-F.},
      author={Zhou, X.},
       title={Occupation times of spectrally negative {L}\'evy processes with
  applications},
        date={2011},
     journal={Stochastic Process. Appl.},
      volume={121},
      number={11},
       pages={2629\ndash 2641},
}

\bib{landriaultetal2014}{article}{
      author={Landriault, D.},
      author={Renaud, J.-F.},
      author={Zhou, X.},
       title={An insurance risk model with {P}arisian implementation delays},
        date={2014},
     journal={Methodol. Comput. Appl. Probab.},
      volume={16},
      number={3},
       pages={583\ndash 607},
}

\bib{loeffenetal2013}{article}{
      author={Loeffen, R.~L.},
      author={Czarna, I.},
      author={Palmowski, Z.},
       title={Parisian ruin probability for spectrally negative {L}\'evy
  processes},
        date={2013},
     journal={Bernoulli},
      volume={19},
      number={2},
       pages={599\ndash 609},
}

\bib{loeffenetal2014}{article}{
      author={Loeffen, R.~L.},
      author={Renaud, J.-F.},
      author={Zhou, X.},
       title={Occupation times of intervals until first passage times for
  spectrally negative {L}\'evy processes},
        date={2014},
     journal={Stochastic Process. Appl.},
      volume={124},
      number={3},
       pages={1408\ndash 1435},
}

\bib{makarov_et_al_2015}{article}{
      author={Makarov, R.},
      author={Metzler, A.},
      author={Ni, Z.},
       title={Modelling default risk with occupation times},
        date={2015},
     journal={Finan. Res. Lett.},
      volume={13},
       pages={54\ndash 65},
}

\bib{Polyetal2008}{book}{
      author={Polyanin, A.~D.},
      author={Manzhirov, A.~V.},
       title={Handbook of integral equations},
     edition={Second},
   publisher={Chapman \& Hall/CRC, Boca Raton, FL},
        date={2008},
}

\bib{takacs1996}{article}{
      author={Tak{\'a}cs, L.},
       title={On a generalization of the arc-sine law},
        date={1996},
     journal={Ann. Appl. Probab.},
      volume={6},
      number={3},
       pages={1035\ndash 1040},
}

\bib{yildirim_2006}{article}{
      author={Yildirim, Y.},
       title={Modelling default risk: a new structural approach},
        date={2006},
     journal={Finan. Res. Lett.},
      volume={3},
      number={3},
       pages={165\ndash 172},
}

\end{biblist}
\end{bibdiv}

\end{document}